\documentclass[11pt]{amsart}
\usepackage[utf8]{inputenc}
\usepackage{amsmath}
\usepackage{amsfonts}
\usepackage{amssymb}
\usepackage{graphicx}
\usepackage{amsmath,amssymb, amsthm}
\usepackage{pgf}
\usepackage{tikz-cd}
\usepackage{xy}
\usepackage{hyperref}
\usetikzlibrary{matrix,arrows,decorations.pathmorphing}
\usepackage[english]{babel}
\usepackage{mathtools}
\usepackage{caption}
\usepackage{geometry}
\usepackage{float}

\theoremstyle{definition}

\theoremstyle{definition}

\theoremstyle{definition}

\theoremstyle{definition}\usepackage{amsmath}

\theoremstyle{plain}
\newtheorem{thm}{Theorem}
\theoremstyle{plain}
\newtheorem{prop}{Proposition}
\theoremstyle{plain}

\theoremstyle{plain}
\newtheorem{lem}{Lemma}

\newtheorem{coro}{Corollary}

\newcommand{\Gal}{\operatorname{Gal}}

\newcommand{\Q}{\mathbb{Q}}

\newcommand{\p}{\mathcal{P}}
\newcommand{\q}{\mathfrak{q}}

\newcommand{\Of}{\mathcal{O}_F}

\newcommand{\epsi}{\varepsilon}

\newcommand{\Nm}{\operatorname{N}}

\newcommand{\pp}{\mathfrak{p}}

\newcommand{\Ol}{\mathcal{O}_L}

\begin{document}

\title[On locally GCD Equivalent fields]{On locally GCD equivalent number fields}

\author[F.~Battistoni]{Francesco Battistoni}
\address{Dipartimento di Matematica\\
         Università di Milano\\
         via Saldini 50\\
         20133 Milano\\
         Italy}
         \email{francesco.battistoni@unimi.it}

\keywords{Equivalence of number fields, density of primes.}

\subjclass[2010]{Primary: 11R16, 11R44, 11R45}

\begin{abstract}
    Local GCD Equivalence is a relation between extensions of number fields which is weaker than the classical arithmetic equivalence. It was originally studied by Lochter with Weak Kronecker Equivalence.\\
    Among the many results he got, Lochter discovered that number fields extensions of degree $\leq 5$ which are locally GCD equivalent are in fact isomorphic. This fact can be restated saying that number fields extensions of low degree are uniquely characterized by the splitting behaviour of a restricted set of primes: in particular, also extensions of degree 3 and 5 are uniquely determined by their inert primes, just like the quadratic fields.\\
    The goal of this note is to present this rigidity result with a different proof, which insists especially on the densities of sets of prime ideals and their use in the classification of number fields up to isomorphism. Alongside Chebotarev's Theorem, no harder tools than basic Group and Galois Theory are required. 
\end{abstract}

\maketitle

\section{Introduction}
Given a number field $F$ and its ring of integers $\mathcal{O}_F$, it is a main topic in Algebraic Number Theory to study the factorization and the splitting type of a rational prime number $p$ in $\mathcal{O}_F$. Some classical questions dealing with this problem are the following: are number fields uniquely determined by the splitting types of rational primes? Is it true that if two number fields share a common set of primes with given splitting type, then the fields are isomorphic?\\
There are plenty of results concerning these questions, especially in the setting of arithmetic equivalence. Two number fields extensions $K/F$ and $L/F$ are said to be \textbf{arithmetically equivalent over $F$} if for almost every prime ideal $\mathfrak{p}\subset\Of$ the splitting types are the same.\\
The following facts, proved by Perlis \cite{perlis} and independent of the base number field $F$, give a strong characterization of arithmetically equivalent extensions:
\begin{itemize}
    \item If two extensions $K/F$ and $L/F$ of degree $\leq 6$ are arithmetically equivalent, then the two fields are $F$-isomorphic.
    \item There exist arithmetically equivalent extensions $K/F$ and $L/F$ of degree 7 which are not $F$- isomorphic.
    \item If $K/F$ and $L/F$ are arithmetically equivalent extensions and one of them is Galois, then they are $F$-isomorphic.
\end{itemize}
This note focuses on a relation which is weaker than arithmetic equivalence: two number fields $K$ and $L$ are \textbf{locally GCD equivalent over a number field $F$} if for every prime $\mathfrak{p}\subset\Of$ which is unramified in both $K$ and $L$ holds
\begin{equation}
    \gcd (f_{1,K}(\mathfrak{p}),\ldots,f_{t,K}(\mathfrak{p})) = \gcd(f_{1,L}(\mathfrak{p}),\ldots,f_{t',L}(\mathfrak{p}))
\end{equation}
where $f_K(\pp):=(f_{1,K}(\mathfrak{p}),\ldots,f_{t,K}(\mathfrak{p})) $ and $f_L(\pp):=(f_{1,L}(\mathfrak{p}),\ldots,f_{t',L}(\mathfrak{p}))$ are the splitting types of $\pp$ in the two fields.
\\
This relation, which is weaker than arithmetic equivalence, forces the number fields involved to have some constraints on their splitting types, and one naturally asks whether the occurring of this equivalence implies the isomorphism or not.
\\
This relation has been called \textit{Local GCD Equivalence} by Linowitz, McReynolds and Miller  \cite{linowitzMcReynoldsMiller}. Nonetheless, it was not a new concept: Lochter \cite{lochterWeakKronecker} already introduced this equivalence (although without giving it a name) and showed that is equivalent to a different relation, called \textit{Weak Kronecker Equivalence}, which was his object of investigation. Lochter's work \cite{lochterWeakKronecker} exploited an approach which consistently relied on Group Theory and representation Theory, and that allowed him to get the following rigidity result.
\begin{thm}\label{LocalGCDEquivalence}
Let $K/F$ and $L/F$ be locally GCD equivalent over $F$ and such that $[K:F], [L:F]\leq 5$. Then $K$ and $L$ are $F$-isomorphic.
\end{thm}

Actually it is possible to translate this statement into a simpler one: in fact, Theorem 1 is equivalent to say that \textit{number fields extensions of degree 2, 3 and 5 are uniquely determined by their inert primes, while number fields extensions of degree 4 are uniquely determined by their inert primes plus the primes with splitting type $(2,2)$}.\\
This alternative expression for Theorem 1 suggests that there could be a way to prove it which is different from Lochter's proof: an idea could be to notice that, if $K$ and $L$ are locally GCD equivalent over $F$, then they have the same splitting types over ``too many primes of a certain kind". Thus one could wonder if it is possible to recover Theorem 1 by means of some results concerning the density of sets of prime numbers: this would be interesting because it would allow to explain a result concerning rigidity properties of number fields by means of simpler tools.\\
The aim of this note is to show that it is indeed possible to prove Theorem 1 using only prime densities and basic Group and Galois Theory: in fact, basic characterizations of the subgroups of symmetric groups $S_n$ (with $n\leq 5$) and well known lemmas of Galois Theory are enough for our purpose.\\
Moreover, although there are many distinct cases to consider for the proof, they can be all solved using mainly two techniques. We specify which technique is used by means of the following notations:
\begin{itemize}
    \item: this symbol denotes the first approach, which consists in reducing the study of two equivalent extensions $K/F$ and $L/F$ at looking for an equivalence of some \textbf{Galois companions} of $K$ and $L$, i.e. some Galois extensions over $F$ which are naturally related to the original fields and have small degree (e.g: if $K/F$ has degree 3 and is not Galois, its Galois closure contains a unique quadratic extension $K_2/F$, which is the companion of $K$).
    \item[**]: this symbol denotes a different approach, which we call \textbf{big Galois closure}: instead of looking for some Galois extension of low degree, one considers a big Galois extension containing both the equivalent extensions $K/F$ and $L/F$, and proves the isomorphism working in this larger setting. We use this technique to deal with the cases where one of the extensions involved is primitive, i.e. has only $F$ and itself as $F$-sub-extensions.\\
\end{itemize}

\subsection{Notation and definitions}
Given a number field extension $L/F$ and an unramified prime $\mathfrak{p}\subset\Of$, its \textbf{splitting type} is the $t$-ple $f_L(\p):=(f_{1,L}(\mathfrak{p}),\ldots,f_{t,L}(\mathfrak{p}))$ given by the inertia degrees $f_{1,L}\leq f_{2,L}\leq\cdots f_{t,L}$ of the prime factors $\mathfrak{q}_1,\ldots,\mathfrak{q}_t$ of $\mathcal{O}_L$ lying over $\mathfrak{p}$.\\
If $\pp$ is unramified and $f_K(\pp)=(1,\ldots,1)$, then $\pp$ is said to be a \textbf{splitting prime.}

Given a set $A$ of prime ideals of $\Of$, its \textbf{prime density} is the number (if it exists)
$$\delta_{\p}(A):=\lim_{x\rightarrow+\infty}\frac{\#\{\mathfrak{p}\in A\colon \Nm(\mathfrak{p})\leq x\}}{\#\{\mathfrak{p}\subset \Of\colon \Nm(\mathfrak{p})\leq x\}} = 
\lim_{x\rightarrow+\infty}\frac{\#\{\mathfrak{p}\in A\colon \Nm(\mathfrak{p})\leq x\}}{x/\log x}.$$
As usual, a property $P$ holds \textbf{for almost all primes in a set $A$ of primes} if $P$ holds for every prime in $A$ up to a subset of $A$ with null prime density (in particular, whenever the set of exceptions is finite).

Given $\mathfrak{p}\subset\Of$, the \textbf{residue field of $\mathfrak{p}$} is the finite field $\mathbb{F}_{\mathfrak{p}}:=\mathcal{O}_F/\mathfrak{p}$.
\\
Given a finite Galois extension $L/F$ with Galois group $G$, an unramified prime $\mathfrak{p}\subset\Of$ and the prime factors $\q_1,\ldots,\q_t$ of $\pp\Ol$, the \textbf{decomposition group of $\q_i$} is the set $G_{\q_i}:=\{\sigma\in G\colon \sigma(\q_i)=\q_i\}$. The $G_{\q_i}$'s are cyclic subgroups of $G$ and are conjugated between them. \\
For every $i=1,\ldots,t$ there is a group isomorphism 
$$\Psi_i : G_{\q_i} \rightarrow \Gal(\mathbb{F}_{\mathfrak{q}_i}/\mathbb{F}_{\mathfrak{p}})= \langle\phi_i \rangle$$
where $\phi_i:\mathbb{F}_{\q_i}\rightarrow \mathbb{F}_{\q_i}$ is the Frobenius automorphism of the finite field $\mathbb{F}_{\q_i}$.\\
The \textbf{Frobenius symbol of $\mathfrak{p}$} is the conjugation class $(L/F,\mathfrak{p}):=\{\Psi_i^{-1}(\phi_i)\colon i=1,\ldots,t\}.$

\subsection{Acknowledgements} The author thanks Harry Smit from University of Utrecht, for several discussions about Local GCD Equivalence, Sandro Bettin from University of Genova for the suggestion that prime densities could be relevant for this problem, and Simone Maletto, for interesting insights about this topics.

\section{Key Lemmas and first characterizations}
\subsection{Technical tools}
Let us begin recalling the only ``heavy" theorem needed, which is the classic Chebotarev's Theorem, necessary for any density argument involving primes in number fields.

\begin{thm}[Chebotarev]
Let $L/F$ be a finite Galois extension of number fields with Galois group $G$. Let $C\subset G$ be a conjugation class in the group. Then, the set of primes $\mathfrak{p}\subset\Of$ such that $(L/F,\mathfrak{p}) = C$ is infinite and has prime density equal to $\#C/\#G$.
\end{thm}

\begin{proof}
See Chapter VIII, Section 4, Theorem 10 of \cite{lang2013algebraic}.
\end{proof}

This theorem allows to compute the densities of primes with given splitting types in Galois extensions of number fields: if we are interested in non-Galois extensions, we use the following proposition.

\begin{prop}\label{propGalois}
Let $E/F$ be a finite Galois number field extension with Galois group $G$, and let $L/F$ be an intermediate extension. Let $H:=\Gal(E/L).$\\
Let $X:=\{H,g_1H\ldots,g_r H\}$ be the set of left cosets of $H$. Let $\mathfrak{p}\subset\Of$ and let $g\in G$ be an element of the Frobenius symbol of $\mathfrak{p}$ in $G$. Consider the action of the group generated by $g$ on $X$ given by left multiplication.\\
Then there is a bijection
$$\{\text{orbits of the action}\}\leftrightarrow \{\text{primes of }\Ol\text{ dividing }\mathfrak{p}\}. $$
Moreover, if $(f_1,\ldots,f_t)$ is the $t$-ple representing the size of the orbits, then $f_L(\mathfrak{p})=(f_1,\ldots,f_t).$
\end{prop}
\begin{proof}
See Chapter III, Prop.2.8 of \cite{janusz}.
\end{proof}
\noindent
Chebotarev's Theorem and Proposition \ref{propGalois} are the tools which allow us to compute the prime densities in number field extensions of degree less or equal than 5 shown in the next sections.
\\

Let us briefly recall a key lemma from Algebraic Number Theory.
\begin{lem}\label{LemmaDecomposizione}
Let $K/F$ and $L/F$ be finite number field extensions and let $KL/F$ be its composite extension. Then $\mathfrak{p}\subset\Of$ splits completely in $KL$ if and only if it splits completely in both $K$ and $L$.
\end{lem}

\begin{proof}
See Chapter III, Prop. 2.5, 2.6 of \cite{janusz}. 
\end{proof}

\begin{coro}\label{coroSplit}
Let $K/F$ be a finite number field extension and let $\widehat{K}/F$ be its Galois closure with group $G$. Then an unramified prime $\mathfrak{p} \subset\Of$ splits completely in $K$ if and only if it splits completely in $\widehat{K}$.
\end{coro}

\begin{proof}
We know that $\widehat{K}$ is the compositum field of all the fields $\sigma(K)$ where $\sigma\in\Gal(L\widehat{K}/F)$; but if a prime splits completely in $K$, it must be totally slit in $\sigma(K)$ as well. The claim follows then from Lemma \ref{LemmaDecomposizione}.
\end{proof}
\noindent

\begin{coro}\label{SplittingGalois}
Let $K/F$ and $L/F$ be finite Galois extensions of number fields and assume that they share the same set of splitting primes (up to exceptions of null prime density). Then $K$ and $L$ coincide.
\end{coro}

\begin{proof}
Let $KL/F$ be the composite Galois extension. By the previous lemma it follows, up to exceptions of null prime density,
$$
\{\mathfrak{p}\subset\Of\colon f_{KL}(\mathfrak{p})=(1,\ldots ,1)\} = \{\mathfrak{p}\subset\Of\colon f_{K}(\mathfrak{p}) =(1,\ldots ,1)\text{ and }f_{L}(\mathfrak{p})=(1,\ldots,1)\}.$$
Applying Chebotarev's Theorem, the identity above gives the equality
$$\frac{1}{[K:F]}= \frac{1}{[KL:F]} = \frac{1}{[L:F]}$$
which immediately implies $K=KL=L.$
\end{proof}

\subsection{Equivalence in degree 2}
We look now at the (easy) study of local GCD equivalence between quadratic fields, and we give also some density result concerning these fields.\\
Remember that the only splitting types available for a quadratic field are $(1,1)$ and $(2).$
\begin{prop}\label{Quadratic}
Let $K$ and $L$ be two quadratic fields over $F$.
\begin{itemize}
    \item[1)] If $K$ and $L$ are locally GCD equivalent over $F$, then they are $F$-isomorphic.
    \item[2)] If $\{\mathfrak{p}\subset\Of\colon f_K(\mathfrak{p})=f_L(\mathfrak{p})=(1,1) \}$ has prime density strictly greater than $1/4$, then $K$ and $L$ are $F$-isomorphic.
    \item[3)] The set $\{\pp\subset\Of\colon f_K(\pp)=f_L(\pp)\}$ has prime density $\geq 1/2$. $K$ and $L$ are equal if and only if the strict inequality holds.
\end{itemize}
\end{prop}

\begin{proof}
\begin{itemize}
    \item[1)] Quadratic extensions over $F$ are Galois extensions: if they are locally GCD equivalent, then they have the same set of splitting primes, and thus they are isomorphic by Corollary \ref{SplittingGalois}.
    \item[2)] Assume that $K\neq L$: then their composite field $KL$ is a Galois field of degree 4 over $F$, and it would be
    $$\{\pp\subset\Of\colon f_{KL}(\pp)=(1,1,1,1)\}=\{\pp\subset\Of\colon f_K(\pp)=f_L(\pp)=(1,1) \}.$$
    But this is a contradiction, since the first set has prime density equal to $1/4$, while the second one has a greater density by the assumption.
    \item[3)] Let $K=F[x]/(x^2-\alpha)$ and $L=F[x]/(x^2-\beta)$, with $\alpha\neq\beta$: the set $\{\pp\subset\Of\colon f_K(\pp)=f_L(\pp)\}$ is identified with the set of splitting primes in $F[x]/(x^2-\alpha \beta)$. The claim follows immediately.
\end{itemize}
\end{proof}

\section{Equivalence in degree 3}
\subsection{Galois Groups for cubic fields}
Let $K$ be a field of degree 3 over $F$, and let $\widehat{K}$ be its Galois closure with Galois group $G$. The group $G$ can be one of the following:
\begin{itemize}
    \item[] $G=C_3$, the cyclic group of order 3. Then $K=\widehat{K}$ is a cubic Galois extension over $F$. The only possible splitting types are $(1,1,1)$ and $(3)$, and furthermore
    
    $$\begin{matrix}
         &\delta_{\p}\{\pp\colon f_K(\pp)=(1,1,1)\} &= &1/3,\\\\
    &\delta_{\p}\{\pp\colon f_K(\pp)=(3)\} &= &2/3.
    \end{matrix}$$

    \item[] $G=S_3$, the symmetric group with 6 elements. Then $\widehat{K}$ has degree 6 over $F$, it contains three $F$-conjugated cubic fields and a quadratic extension $K_2/F$.\\
    Furthermore there are infinitely many primes with splitting type $(1,2)$, each one having Frobenius symbol equal to the elements of order 2 in $S_3$.\\
    Looking at the densities in detail, one has:
    $$\begin{matrix}
         &\delta_{\pp}\{p\colon f_K(\pp)=(1,1,1)\} &= &1/6,\\\\
    &\delta_{\p}\{\pp\colon f_K(\pp)=(1,2)\} &= &1/2,\\\\\
    &\delta_{\p}\{\pp\colon f_K(\pp)=(3)\} &= &1/3.
    \end{matrix}$$
    \end{itemize}
All density computations are derived from Chebotarev's Theorem and Proposition \ref{propGalois}. 

\subsection{Locally GCD equivalent cubic fields}
The equivalence problem in this degree can be solved by means of the sole Galoic companions technique. 
\\
Let $K$ and $L$ be two cubic fields over $F$ which are locally GCD equivalent.
\begin{itemize}
\item It is almost immediate to see that if one of them (assume $K$) is Galois, then the other extension is Galois, because of the density of the inert primes. But if $K/F$ and $L/F$ are Galois cubic extensions and are locally GCD equivalent, they have the same splitting primes, and thus $K=L$.\\
\item Let us assume that both $K$ and $L$ are not Galois. Consider their Galois closures $\widehat{K}$ and $\widehat{L}$, and the quadratic Galois companions $K_2$ and $L_2$.\\
Using Proposition \ref{propGalois}, it is easy to show the following correspondence among the splitting types of the fields involved:
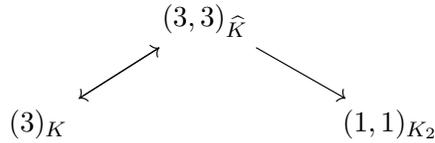
\begin{displaymath}
        \begin{tikzcd}
        & & (3, 3)_{\widehat{K}}\arrow{ld}\arrow{rd} &\\
        &(3)_K\arrow{ru} & &(1,1)_{K_2}
        \end{tikzcd}
        \end{displaymath}
One gets the following identity:
\begin{equation}\label{Caso 3}
\{\pp\colon f_{K_2}(\pp)=(1,1), f_K(\pp)=(3)\} = \{\pp\colon f_{L_2}(\pp)=(1,1), f_L(\pp)=(3)\}.
\end{equation}
This implies that $\{\pp\colon f_{K_2}(\pp)=(1,1)=f_{L_2}(\pp)\}$ has prime density greater than $1/3$, and by Proposition \ref{Quadratic} one has $K_2=L_2$.\\
The remaining splitting primes in $K_2$, which have prime density equal to $1/2 - 1/3 = 1/6$, are exactly the ones that split completely in $\widehat{K}$. But this fact, together with $K_2=L_2$ and Equality (\ref{Caso 3}), force $\widehat{K}$ and $\widehat{L}$ to have the same splitting primes, i.e. $\widehat{K}=\widehat{L}$, which in turn implies $K\simeq L$ (because the cubic extensions in $\widehat{K}/F$ are $F$-conjugated between them).
\end{itemize}

\section{Equivalence in degree 4}
\subsection{Galois groups for quartic fields} Let $K$ be a field of degree 4 over $F$, and let $\widehat{K}$ be its Galois closure with Galois group $G$. The group $G$ can be one of the following:
\begin{itemize}
    \item[] $G=C_4$, the cyclic group of order 4. Then $K=\widehat{K}$ is Galois over $F$ and the splitting types and densities are as follows:
       $$\begin{matrix}
         &\delta_{\p}\{\pp\colon f_K(\pp)=(1,1,1,1)\} &= &1/4,\\\\
    &\delta_{\p}\{\pp\colon f_K(\pp)=(2,2)\} &= &1/4,\\\\\
    &\delta_{\p}\{\pp\colon f_K(\pp)=(4)\} &= &1/2.
    \end{matrix}$$
    
    \item[] $G=C_2\times C_2$: then $K=\widehat{K}$ is Galois over $F$ and 
       $$\begin{matrix}
         &\delta_{\p}\{\pp\colon f_K(\pp)=(1,1,1,1)\} &= &1/4,\\\\
    &\delta_{\p}\{\pp\colon f_K(\pp)=(2,2)\} &= &3/4.
    \end{matrix}$$
    
    \item[] $G=D_4:=\langle \sigma,\tau | \sigma^4 =\tau^2 =1, \tau\sigma\tau =\sigma^3\rangle$. Then $\widehat{K}$ has degree 8 over $F$, it contains 5 quartic fields and 3 quadratic fields, the lattice of sub-extensions being as follows:
    
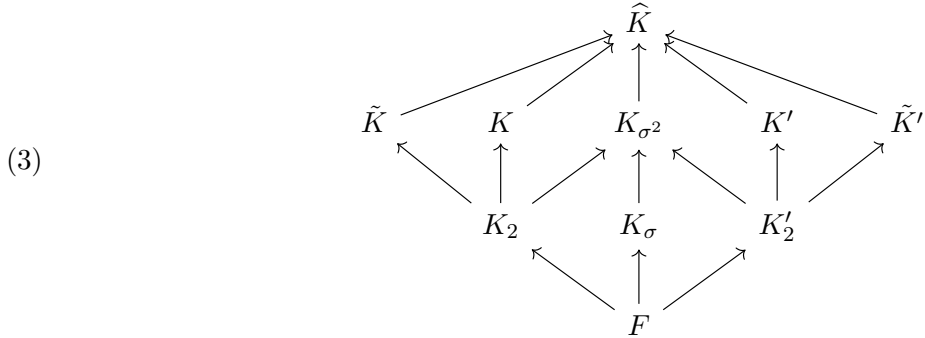
\begin{equation}\label{latticeD4}
    \begin{tikzcd} 
    & & &\widehat{K} & &\\
    &\Tilde{K}\arrow{rru} &K\arrow{ru} & K_{\sigma^2}\arrow{u} & K'\arrow{lu} & \Tilde{K'}\arrow{llu}\\
    & &K_2\arrow{lu}\arrow{u}\arrow{ru} & K_{\sigma}\arrow{u} & K_2'\arrow{lu}\arrow{u}\arrow{ru} &\\
    & & &F\arrow{lu}\arrow{u}\arrow{ru} & &
    \end{tikzcd}
    \end{equation}
    The quartic fields form 3 distinct classes of $F$-isomorphism: $\{K,\Tilde{K}\}$, $\{K', \Tilde{K'}\}$ and $\{K_{\sigma^2}\}$. The extension $K_{\sigma^2}/F$ is Galois with Galois group $C_2\times C_2$.\\
    Assuming that $\Gal(\widehat{K}/K)=\langle\tau\rangle$, Proposition \ref{propGalois} yields
    $$\begin{matrix}
         &\delta_{\p}\{\pp\colon f_K(\pp)=(1,1,1,1)\} = \delta_{\p}\{\pp\colon (\widehat{K}/F, \pp)=1_{D_4}\} &= &1/8,\\\\
    &\delta_{\p}\{\pp\colon f_K(p)=(1,1,2)\} = \delta_{\p}\{\pp\colon (\widehat{K}/F, \pp)=\{\tau,\sigma^2\tau\}\} &= &1/4,\\\\
    &\delta_{\p}\{\pp\colon f_K(p)=(2,2), (\widehat{K}/F, \pp)=\sigma^2 \} &= &1/8,\\\\
     &\delta_{\p}\{\pp\colon f_K(p)=(2,2), (\widehat{K}/F, \pp)=\{\sigma\tau, \tau\sigma\} \} &= &1/4,\\\\
     &\delta_{\p}\{\pp\colon f_K(p)=(4)\} = \delta_{\p}\{\pp\colon (\widehat{K}/F, \pp)=\{\sigma,\sigma^3\}\} &= &1/4.
    \end{matrix}$$
    If $\Gal(\widehat{K}/K)=\langle \sigma\tau \rangle$, simply reverse the roles of $\tau$ and $\sigma\tau$ in the description above.
    \\
    \item[] $G=A_4$, the alternating group with 12 elements. Then $\widehat{K}$ has degree 12 over $F$, it contains $4$ quartic fields (each one $F$-conjugated to the others) and a Galois cubic extension $K_3/F$. There are no inert primes, and the splitting types and densities are the following:
       $$\begin{matrix}
         &\delta_{\p}\{\pp\colon f_K(\pp)=(1,1,1,1)\} &= &1/12,\\\\
    &\delta_{\p}\{\pp\colon f_K(\pp)=(1,3)\} &= &2/3,\\\\\
    &\delta_{\p}\{\pp\colon f_K(\pp)=(2,2)\} &= &1/4.
    \end{matrix}$$
    \\
    \item[] $G=S_4$: then $\widehat{K}$ has degree 24 over $F$ and contains a Galois extension $K_6/F$ of degree 6, while all the quartic extensions over $F$ contained in $\widehat{K}$ are $F$-conjugated. The splitting types and decomposition are as follows:
    
           $$\begin{matrix}
         &\delta_{\p}\{\pp\colon f_K(\pp)=(1,1,1,1)\} &= &1/24,\\\\
         &\delta_{\p}\{\pp\colon f_K(\pp)=(1,1,2)\} &= &1/4,\\\\
    &\delta_{\p}\{\pp\colon f_K(\pp)=(1,3)\} &= &1/3,\\\\\
    &\delta_{\p}\{\pp\colon f_K(\pp)=(2,2)\} &= &1/8,\\\\
    &\delta_{\p}\{\pp\colon f_K(\pp)=(4)\} &= &1/4.\\\\
    \end{matrix}$$
\end{itemize}
    All densities are computed with Proposition \ref{propGalois}.\\
    These data immediately show that if $K/F$ and $L/F$ are locally GCD equivalent quartic extensions, then they must have the same Galois closure.
    
\subsection{Locally GCD equivalent quartic fields}  
Just like for the previous degree, searching for Galois companions will be enough to study the equivalence between extensions of degree 4.
\\
As mentioned before, we only study locally GCD equivalent quartic extensions $K/F$ and $L/F$ with same Galois group. This immediately implies that whenever one of the extensions is Galois, then the equivalence is actually an isomorphism.\\

    \begin{itemize}
          \item $G=A_4$: Consider the cubic Galois companions $K_3/F$ and $L_3/F$ associated to $K$ and $L$ respectively. Proposition \ref{propGalois} yields the following behaviour on the splitting types: 
        
         \begin{displaymath}
        \begin{tikzcd}
        & & (2\times 6)_{\widehat{K}}\arrow{ld}\arrow{rd} &\\
        &(2,2)_K\arrow{ru} & &(1,1,1)_{K_3}
        \end{tikzcd}
        \end{displaymath}
        
        Thus one gets the identity
        \begin{equation}\label{A4}
            \{\pp\colon f_{K_3}(\pp)=(1,1,1), f_K(\pp)=(2,2)\} = \{\pp\colon f_{L_3}(\pp)=(1,1,1), f_L(\pp)=(2,2)\}.
        \end{equation}
        The sets above have prime density $1/4$, and this forces $K_3=L_3$; if this was not true, the composite Galois extension $KL/F$ would have degree 9. But being
          $$\{\pp\colon f_{K_3L_3}(\pp)=(1\times 9)\}=\{\pp\colon f_{K_3}(\pp)=f_{L_3}(\pp)=(1,1,1) \},$$
        the left hand side would have prime density equal to $1/9$, which is in contradiction with Equality (\ref{A4}).\\
        The remaining splitting primes in $K_3$ have density $1/3-1/4=1/12$ and are precisely the primes which split completely in the Galois closure $\widehat{K}$. Thus, equality (\ref{A4}) and $K_3=L_3$ force $\widehat{K}$ and $\widehat{L}$ to have the same splitting primes, i.e. $\widehat{K}=\widehat{L}$, which implies $K\simeq L$.\\
        \item The case $G=S_4$ is completely similar: one associates to $K$ the unique Galois sextic extension $K_6/F$ contained in $\widehat{K}$, and using the densities of primes $\pp$ with $f_K(\pp)=(2,2)$ one forces $K_6=L_6$ and from that $\widehat{K}=\widehat{L}$, which in turn gives $K\simeq L$.\\

        \item $G=D_4$: Let us take $K/F$ and $L/F$ locally GCD equivalent quartic extensions with Galois closures $\widehat{K}$ and $\widehat{L}$ and Galois group $D_4$. We follow the notations of diagram (\ref{latticeD4}) for the sub-extensions of $\widehat{K}$ and $\widehat{L}$.\\\\
        Consider the subfield $K_2\subset K$: it is immediate to see that, if $f_K(\pp)=(4)$, then $f_{K_2}(\pp)=(2)$; in the same way, a prime ideal $\pp$ such that $f_K(\pp)\in\{(1,1,1,1),(1,1,2)\}$ has splitting type $f_{K_2}(\pp)=(1,1)$. These facts, together with the local GCD equivalence between $K$ and $L$, yield the equalities:
        \begin{equation}\label{D4caso1}
            \{\pp\colon f_{K_2}(\pp)=(2), f_K(\pp)=(4)\} = \{\pp\colon f_{L_2}(\pp)=(2), f_L(\pp)=(4)\},
        \end{equation}
        \begin{align}\label{D4caso2}
            &\{\pp\colon f_{K_2}(\pp)=(1,1), f_K(\pp)\in\{(1,1,1,1),(1,1,2)\}\} =\nonumber\\
            &\{\pp\colon f_{L_2}(\pp)=(1,1), f_L(\pp)\in\{(1,1,1,1),(1,1,2)\}\}.
        \end{align}
        The sets in Equality (\ref{D4caso1}) have prime density equal to $1/4$, while the ones in Equality (\ref{D4caso2}) have prime density equal to $3/8$. This tells us that $K_2$ and $L_2$ have the same splitting type on at least $5/8$ of the primes, and so $K_2=L_2$ by Proposition \ref{Quadratic}.\\\\
       Let us consider now the field $K_{\sigma}$. Using Proposition \ref{propGalois}, it is possible to show the following behaviour:
       
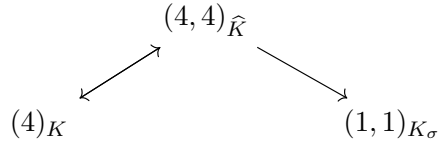
\begin{displaymath}
        \begin{tikzcd}
        & & (4, 4)_{\widehat{K}}\arrow{ld}\arrow{rd} &\\
        &(4)_K\arrow{ru} & &(1,1)_{K_{\sigma}}
        \end{tikzcd}
        \end{displaymath}
        
       Thus one obtains the equality
       \begin{equation}\label{D4splitting}
            \{\pp\colon f_{K_{\sigma}}(\pp)=(1,1), f_K(\pp)=(4)\} = \{\pp\colon f_{L_{\sigma}}(\pp)=(1,1), f_L(\pp)=(4)\}
       \end{equation}
       and the sets above have prime density equal to $1/4$.\\ 
       Furthermore, the set of primes $\{\pp\colon f_K(\pp)=f_L(\pp)=(1,1,1,1)\}$ has positive density $\epsi>0$ (because it corresponds to the set of splitting primes in the composite extension $K L$) and, thanks to the fact that these primes split completely also in $\widehat{K}$ and $\widehat{L}$, it is clear that for any of these primes holds $f_{K_{\sigma}}(\pp)=f_{L_{\sigma}}(\pp)=(1,1).$ This result, together with Equality (\ref{D4splitting}), yields $K_{\sigma}=L_{\sigma}$, and together with $K_2=L_2$ provides $K_{\sigma^2}=L_{\sigma^2}.$
                \\
                Now, we show that $\widehat{K}=\widehat{L}$: one has the equalities
                $$\{\pp\colon f_K(\pp)=(2,2)\} = \{\pp\colon  f_L(\pp)=(2,2)\},$$
                
                $$\{\pp\colon f_{K_{\sigma^2}}(\pp)=(1,1,1,1)\} = f_{L_{\sigma^2}}(\pp)=(1,1,1,1)\}$$
                and the intersection of these sets gives
                $$\{\pp\colon f_{K_{\sigma^2}}(\pp)=(1,1,1,1), f_K(\pp)=(2,2)\} = \{\pp\colon f_{L_{\sigma^2}}(\pp)=(1,1,1,1), f_L(\pp)=(2,2)\}.$$
                The sets above have prime density exactly equal to $1/8$, because they are the primes with $\sigma^2$ as Frobenius symbol. This means that the remaining splitting primes in $K_{\sigma^2}$, which have prime density equal to $1/4-1/8 = 1/8$, identify $\widehat{K}$; but being $K_{\sigma^2}=L_{\sigma^2}$, this means that $\widehat{K}$ and $\widehat{L}$ have the same splitting primes, i.e. $\widehat{K}=\widehat{L}$.\\
                Finally, we show that $K\simeq L$: if they were not, it would be $L\simeq K'$; but then $K$ and $L$ could not be locally GCD equivalent, because a prime with Frobenius symbol $\langle\tau\rangle$ would have splitting type $(2,2)$ in one field but $(1,1,2)$ in the other.\\
    \end{itemize}
\section{Equivalence in degree 5}
\subsection{Galois groups for quintic fields}
Let $K$ be a field of degree 5 over $F$. The following are the possibilities for the Galois group $G$ of its Galois closure $\widehat{K}$. We shall focus mainly on the set of inert primes and its density.

\begin{itemize}
    \item[] $G=C_5$, the cyclic group of order 5. Then $\widehat{K}=K$ and
     $$\begin{matrix}
         &\delta_{\p}\{\pp\colon f_K(\pp)=(1,1,1,1,1)\} &= &1/5,\\\\
    &\delta_{\p}\{\pp\colon f_K(\pp)=(5)\} &= &4/5.
    \end{matrix}$$
    \\
    \item[] $G=D_5:=\langle\sigma,\tau | \sigma^5=\tau^2=1, \tau\sigma\tau =\sigma^{-1}\rangle.$ Then $\widehat{K}$ has degree 10 over $F$, contains 5 $F$-conjugated quintic fields and a unique quadratic extension $K_2/F$. Moreover:
     $$\begin{matrix}
         &\delta_{\p}\{\pp\colon f_K(\pp)=(1,1,1,1,1)\} &= &1/10,\\\\
         &\delta_{\p}\{\pp\colon f_K(\pp)=(1,2,2)\} &= &1/2,\\\\
    &\delta_{\p}\{\pp\colon f_K(\pp)=(5)\} &= &2/5.
    \end{matrix}$$
    \\
    \item[] $G=F_5:=\langle \sigma, \mu | \sigma^4=\mu^5=1, \mu\sigma= \sigma\mu^2\rangle$. Then $\widehat{K}$ has degree 20 over $F$, contains 5 $F$-conjugated quintic fields and a unique Galois, cyclic quartic extension $K_4/F$, and furthermore:
     $$\begin{matrix}
         &\delta_{\p}\{p\colon f_K(p)=(1,1,1,1,1)\} &= &1/20,\\\\
         &\delta_{\p}\{p\colon f_K(p)=(1,4)\} &= &3/4,\\\\
    &\delta_{\p}\{p\colon f_K(p)=(5)\} &= &1/5.
    \end{matrix}$$
    \\
    \item[] $G=A_5$, the alternating group with 60 elements. Then $\widehat{K}$ has degree 60 over $F$ and, most importantly, there are no non-trivial Galois $F$-extensions in it. The quintic fields in $\widehat{K}$ are all $F$-conjugated, and by Corollary \ref{SplittingGalois} every non-trivial subfield has the same splitting primes of $K$, implying that $\widehat{K}$ is uniquely determined by one of its non-trivial $F$-sub-extensions. Looking only at the inert primes, one gets:
     $$\begin{matrix}
    &\delta_{\p}\{\pp\colon f_K(\pp)=(5)\} &= &2/5.
    \end{matrix}$$
    \\
    \item[] $G=S_5$, the symmetric group with 120 elements. Then $\widehat{K}$ has degree 120 over $F$, its only Galois $F$-subfields being $F/F$ and a quadratic extension $K_2/F$. Every other $F$-sub-extension is non-Galois and shares with $\widehat{K}$ the set of splitting primes. The quintic subfields are $F$-conjugated. The inert primes satisfy:
     $$\begin{matrix}
          &\delta_{\p}\{\pp\colon f_K(\pp)=(5)\} &= &1/5.
    \end{matrix}$$
\end{itemize}

\subsection{Locally GCD equivalent quintic fields} Degree 5 extensions are the first one which present cases of primitive, non Galois extensions. Whenever one of these extensions occur, we will use the Big Galois Closure approach instead of the Galois companions.
\\
Let $K$ and $L$ be locally GCD equivalent fields of degree 5 over $F$. It is immediate from the density of the inert primes that, if one of them is Galois over $F$, then the two fields are actually isomorphic. Moreover, if $\widehat{K}$ has group $G_K=D_5$, then $\widehat{L}$ has group $G_L$ equal to either $D_5$ or $A_5$; if $\widehat{K}$ has $G_K=F_5$, then $\widehat{L}$ has group $G_L$ equal to either $F_5$ or $S_5$.
\\
\begin{itemize}
    \item $G_K=D_5$ and $G_L=D_5$: let $K_2/F$ and $L_2/F$ be the quadratic Galois companions of $K$ and $L$ respectively. Proposition \ref{propGalois} yield the following behaviour on inert primes:
    
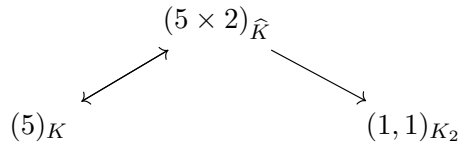
\begin{displaymath}
        \begin{tikzcd}
        & & (5\times 2)_{\widehat{K}}\arrow{ld}\arrow{rd} &\\
        &(5)_K\arrow{ru} & &(1,1)_{K_2}
        \end{tikzcd}
        \end{displaymath}
   Thus one has the identity
    \begin{equation}\label{D5D5}
         \{\pp\colon f_{K_2}(\pp)=(1,1), f_K(\pp)=(5)\} = \{\pp\colon f_{L_2}(\pp)=(1,1), f_L(\pp)=(5)\}.
    \end{equation}
    The above set has prime density equal to $2/5 > 1/4$, and this implies $K_2=L_2$ by Proposition \ref{Quadratic}.\\
    The remaining splitting primes in $K_2$ (which have density $1/2-2/5=1/10$) are precisely the primes which split completely in $\widehat{K}$. Thus Equality (\ref{D5D5}) and $K_2=L_2$ force $\widehat{K}$ and $\widehat{L}$ to have the same splitting primes, i.e. $\widehat{K}=\widehat{L}$. This yields $K\simeq L$.\\
    
    \item $G_K=F_5$ and $G_L=F_5$: this case is completely simlar to the previous one: just consider the quartic Galois companions $K_4$ and $L_4$ of $K$ and $L$ respectively. The inert primes of $K$ become splitting primes of $K_4$: like before, this forces $K_4=L_4$ and from that one gets $\widehat{K}=\widehat{L}$, which in turn gives $K\simeq L$.
   \\
    \item[**] $G_K=A_5$ and $G_L=A_5$: consider the Galois closures $\widehat{K}$ and $\widehat{L}$ and let us study their intersection.\\
    If $\widehat{K}\cap\widehat{L}$ is different from $F$, then there is a common non-trivial subfield, which identifies the same splitting primes for both the fields, implying $\widehat{K}=\widehat{L}$ and $K\simeq L$.\\
    So assume the intersection is equal to $F$: the composite Galois extension $\widehat{K}\widehat{L}$ has degree 3600 and Galois group $A_5\times A_5$. A prime $\pp$ which is inert in both $K$ and $L$ has a Frobenius symbol formed by elements of order 5 in $A_5\times A_5$. These elements have the form $(g,h)$ with $g^5=h^5=1_{A_5}$, with the only exception of $g=h=1_{A_5}$.\\
    But by local GCD equivalence, the set of such primes has prime density $2/5$, while the density of the primes having elements of order 5 in $A_5\times A_5$ as Frobenius symbols is $(25\cdot 25-1)/3600 = 624/3600 < 1/4 <2/5$, which is a contradiction.
    \\
    \item[**] We are left with the cases $G_K=S_5$ and $G_L=S_5$,  $G_K=D_5$ and $G_L=A_5$ and the case $G_K=F_5$ and $G_K=S_5$. These cases are solved by using the Big Galois Closure technique: one studies the intersection between $\widehat{K}$ and $\widehat{L}$ and must distinguish between two cases: if $\widehat{K}\cap \widehat{L}$ has degree greater or equal than 5, then the intersection is a field which uniquely detects both $\widehat{K}$ and $\widehat{L}$, forcing $\widehat{K}=\widehat{L}$ and thus the isomorphism between $K$ and $L$. If $[\widehat{K}\cap\widehat{L}:F]\leq 2$ instead (no degree 3 or 4 intersection occurs) then one imitates the proof of the case $G_K=A_5, G_L=A_5$ in order to get a composite field $\widehat{K}\widehat{L}$ with a degree so large that the density of order 5 Frobenius elements in $G_K\times G_L$ results strictly less then the the product of the densities of inert primes in $K$ and $L$, while the two things should be equal.

\end{itemize}

\section{Final Remarks}
\subsection{Comparing equivalent fields of different degree}
    The proofs in the previous section showed that any two number field extensions having same degree $n\leq 5$ which are locally GCD equivalent are in fact isomorphic. In order to complete the proof of Theorem \ref{LocalGCDEquivalence}, one needs to see what happens when one compares equivalent fields of different degrees.\\
    The prime densities computations of the previous sections show that this possibility cannot exist for locally GCD equivalent fields of degree $n\leq 5$: among the field extensions with these degrees, cubic fields can be equivalent (and thus isomorphic) only to cubic fields, because the inert primes have greatest common divisor of their splitting type equal to 3, a number which is not obtained in any other low degree. For the same reason, quintic fields can be equivalent only to quintic fields.\\
    We are left only with the comparison between quadratic and quartic extensions; but in any quadratic extension the inert primes have density $1/2$, while in quartic fields such a density value is not attained by primes with splitting type $(2,2)$.
    
\subsection{A counterexample in degree 6}    
    Theorem \ref{LocalGCDEquivalence} proves that the local GCD equivalence reduces to isomorphism on equivalent fields of degree $n\leq 5$. It can be proven that there are counterexamples already in degree 6 : in fact, for every Galois cubic extension $K/F$, it is possible to present two non isomorphic quadratic extensions $L/K$ and $M/K$ such that $L/F$ and $M/F$ are $F$-locally GCD equivalent extensions of degree 6.\\
    The construction relies on two concepts: first, local GCD equivalence can be proved to be equivalent to the fact that the norm groups of the fractional ideals are the same for the two extensions (see \cite{arithmeticalSimilarities}, Chapter VI, Section 1.b for the details). Then, using this different formulation, Stern \cite{stern} proved the existence of the sextic extensions $L/F$ and $M/F$ as above.
    \\
    Moreover, being the much stronger relation given by arithmetic equivalence not reducible to the isomorphism for degrees $n\geq 7$, we can finally state that $5$ is the maximum degree $n$ for which the claim of Theorem \ref{LocalGCDEquivalence} hold for every number field extension of degree $n$.
    
\subsection{Inert primes are not enough in quartic fields}
        As previously reported, Theorem \ref{LocalGCDEquivalence} can be expressed, for number fields extensions of prime degree $p\leq 5$, by saying that these extensions are uniquely determined by their inert primes. This formulation, although very elementary, has no direct references in literature: in fact, a proof of this result for cubic fields was the original reason for the author to start studying this subject, and which in the end led him to Lochter's paper \cite{lochterWeakKronecker}.
    \\
         One could wonder if also quartic fields are uniquely determined by their inert primes, in the cases for which they actually exist.
    This request is much weaker than local GCD equivalence, and, as we show below, it is not enough in order to have an isomorphism.\\
    In fact, there are easy counterexamples: take a quartic field $K$ with Galois closure $\widehat{K}$ having Galois group $D_4$ and consider the non-conjugated non-Galois field $K'$ contained in $\widehat{K}$ (refer to diagram \ref{latticeD4} for notations). Then a prime $\pp\subset\Of$ is inert in $K$ if and only if its Artin symbol in $D_4$ is formed by elements of order 4: but the computations given by Proposition \ref{propGalois} show that the very same property holds also for $K'$, and so we have two non-isomorphic quartic field extensions with same inert primes.
    \\
    As an explicit example, consider $K:= \Q[x]/(x^4-3x^2-3)$ and $K':=\Q[x]/(x^4-3x+3)$: these quartic fields are not Galois over $\Q$ and share the same Galois closure over $\Q$, which is the octic field $\widehat{K}:=\Q[x]/(x^8 + x^6 - 3x^4 + x^2 + 1)$ with Galois group $D_4$; so they share the inert primes, but in fact $K$ and $K'$ are not isomorphic.
    
\subsection{Similar results in higher degree}    
Although 5 is the maximum degree for which Theorem \ref{LocalGCDEquivalence} holds, it is still possible to get a similar rigidity result for large families of field extensions in arbitrary prime degree by a simple adaptation of the Big Galois Closure technique used previously.\\
 Let $p$ be a prime number. Let $K/F$ be a number field extension of degree $p$, and assume that its Galois closure has group equal to either $A_p$ or $S_p$. Applying Proposition \ref{propGalois} it is easy to prove that this field has inert primes. 
    If one mimics the procedure used to reduce the equivalence of quintic fields having group $A_5$ or $S_5$ to isomorphism, then it is possible to get the following theorem.
    \begin{thm}\label{LocalGCDEquivalencePrimeDegree}
    Let $K$ and $L$ be number fields of prime degree $p$ over $F$ which are $F$-locally GCD Equivalent and such that their Galois closures share the same Galois group $G$. Assume $G$ equal either to $A_p$ or $S_p$  Then $K$ and $L$ are $F$-isomorphic.
    \end{thm}
    
    Theorem \ref{LocalGCDEquivalencePrimeDegree} is actually very strong, because of the fact that a ``random" number field extension of prime degree tends to have Galois group of its closure equal to the symmetric group $S_p$: from this one can conclude that, for these degrees, the local GCD equivalence reduces very often to isomorphism.\\
    A stronger result, always by Lochter, proves Theorem \ref{LocalGCDEquivalencePrimeDegree} for every degree $n$ and Galois groups $S_n$ and $A_n$. At the moment, it seems not reachable without the group-theoretic setting, or by means of the Big Galois Closure technique alone.

\end{document}